\documentclass[11pt,letterpaper]{amsart}

\begin{document}


\theoremstyle{plain}
\newtheorem*{Claim}{Claim}
\newtheorem{theorem}{Theorem} [section]
\newtheorem{corollary}[theorem]{Corollary}
\newtheorem{lemma}[theorem]{Lemma}
\newtheorem{proposition}[theorem]{Proposition}
\newtheorem{thmx}{Theorem}
\renewcommand{\thethmx}{\Alph{thmx}}


\theoremstyle{definition}
\newtheorem{definition}[theorem]{Definition}

\theoremstyle{remark}
\newtheorem{remark}[theorem]{Remark}

\numberwithin{theorem}{section}
\numberwithin{equation}{section}
\numberwithin{figure}{section}

\newcommand{\cle}{\left[}
\newcommand{\cri}{\right]}
\newcommand{\pl}{\left(}
\newcommand{\pr}{\right)}
\newcommand{\Rc}{\text{\rm Rc}}  
\newcommand{\Rm}{\text{\rm Rm}}  

\newcommand{\Sph}{\mathbb{S}}
\newcommand{\R}{\mathbb{R}}
\newcommand{\D}{\nabla}
\newcommand{\He}{\text{\rm Hess}}
\newcommand{\la}{\langle}
\newcommand{\ra}{\rangle}
\newcommand{\Si}{\Sigma}
\newcommand{\dint}{\displaystyle\int}
\newcommand{\dvol}{d\text{\rm vol}}

\newcommand{\dsigma}{\text{\rm d}\sigma}
\newcommand{\dt}{\frac{\text{\rm d}}{\text{\rm d}t}}
\newcommand{\dr}{\frac{\text{\rm d}}{\text{\rm d}r}}
\newcommand{\ddt}{\frac{\text{\rm d}^2}{\text{\rm d}t^2}}
\newcommand{\ddtzero}{\left.\frac{\text{\rm d}^2}{\text{\rm d}t^2}\right|_{t=0}}
\newcommand{\Ham}{\text{\rm H}}
\newcommand{\rhobt}{\overline{\rho_t}}

\title[Almost-rigidity and the extinction time]{Almost-rigidity and the extinction time of positively curved Ricci flows}

\author[R. Bamler]{Richard H. Bamler}

\address{Department of Mathematics, University of California, Berkeley}

\email{rbamler@math.berkeley.edu}

\author[D. Maximo]{Davi Maximo}

\address{Department of Mathematics, Stanford University}

\email{maximo@math.stanford.edu}

\maketitle

\begin{abstract}
We prove that Ricci flows with almost maximal extinction time must be nearly round, provided that they have positive isotropic curvature when crossed with $\mathbb{R}^{2}$. As an application, we show that positively curved metrics on $S^{3}$ and $RP^{3}$ with almost maximal width must be nearly round. 
\end{abstract}

\section{Introduction}

Let $g(t)$ be a Ricci flow with $g(0)=g_0$, where $g_0$ is a metric with scalar curvature $R_{g_0} \geq n(n-1)$ and such that  $(M,g_0)\times\mathbb{R}^2$ has positive isotropic curvature. By the work of Brendle-Schoen \cite{BreSch09}, $g(t)$ converges in finite time, say $T$, to a round point (more precisely, the rescaled metrics $\frac{1}{2(n-1)(T-t)}g(t)$ converge to a metric of constant curvature one as $t\rightarrow T$ and, therefore, $M$ must be diffeomorphic to a spherical space form).  Since the  scalar curvature $R(x,t)$ satisfies 
\begin{equation}\label{eq:scalar}
\partial_t R = \Delta R + 2|\textrm{Ric}|^2 \geq \Delta R + \frac{2}{n}R^2,
\end{equation}
one has, by the maximum principle, that $\inf R(\cdot,t) \geq \frac{n(n-1)}{1-2(n-1)t}$. From this, the singular time $T$ can be estimated from above, $i.e.,$ $T\leq \frac{1}{2(n-1)}$. 
Another consequence of $\eqref{eq:scalar}$ is that $T=\frac{1}{2(n-1)}$ if, and only if,  $g_0$ is an Einstein manifold with $R \equiv n(n-1)$.
So $g(t) = (1-2(n-1)t) g_0$, which implies that $g_0$ is round.
In this work, we are interested in flows where $T$ is very close to $\frac{1}{2(n-1)}$. 
Heuristically, from evolution \eqref{eq:scalar}, a metric with almost maximal extinction time should have small traceless Ricci tensor and therefore be nearly Einstein. We prove:

\begin{theorem}\label{thmmain}
Let $(M^{n},g_{0})$, $n\geq3$, be a Riemannian manifold such that $R_{g_{0}}\geq n(n-1)$ and $(M^{n},g_{0})\times\mathbb{R}^{2}$ has positive isotropic curvature. Given $\eta>0$, there exists a number $\tau>0$, which only depends on $\eta$ and on the topology of $M$, such that if the Ricci flow evolution $g(t)$ of $g_{0}$ has singular time $T>\frac{1}{2(n-1)}-\tau$, then $g_{0}$ is $\eta$-close to a metric of constant curvature one in the $\mathcal{C}^{0}$-norm. 
\end{theorem}

Note that the conclusion of the theorem is essentially optimal.
In fact, given any $\tau > 0$, we can use the results of \cite{Sim02} or \cite{KoLa13} to find a constant $\eta > 0$ such that any $\eta$-perturbation of the round metric on $M$ has extinction time $T > \frac1{2(n-1)} - \tau$.

Theorem \ref{thmmain} was motived by the following application to 3-dimensional manifolds. Let $g$ be a metric on the 3-sphere $S^{3}$  with positive sectional curvature and scalar curvature $R\geq 6$ (in particular, $g$ satisfies the hypothesis of Theorem \ref{thmmain}). We recall the definition of the {\it width} of $g$: starting with the family $\{\overline{\Sigma}_{t}\}$ of level sets of the height function $x_{4}:S^{3}\subset\mathbb{R}^{4}\rightarrow\mathbb{R}$; $i.e.$, 
$$\{\overline{\Sigma}_{t}\}=\{x\in S^{3}\,|\,x_{4}=t\}$$
for $t\in[-1,1]$, we define $\overline{\Lambda}$ to be the collection of all families $\{\Sigma_{t}\}$ with the
property that $\Sigma_{t}=F_{t}(\overline{\Sigma}_{t})$ for some smooth one-parameter family of diffeomorphisms $F_{t}$ of $S^{3}$,
 all of which are isotopic to the identity. The {width} of $(S^{3},g)$ is the min-max invariant
$$W(g)=\displaystyle \inf_{\{\Sigma_{t}\in\overline{\Lambda}\}}\sup_{t\in[-1,1]} |\Sigma_{t}|,$$
where $|\Sigma|$ denotes the surface area of $\Sigma$. In \cite{MaNe12}, Marques-Neves proved that there exists an embedded minimal sphere $\Sigma$, of Morse index one, such that
$$W(g)=|\Sigma|\leq 4\pi,$$
and equality $W(g)=4\pi$ holds if, and only if, $g$ has constant sectional curvature one. Using Theorem \ref{thmmain}, we investigate the almost equality case:

\begin{theorem}\label{thmapp}
Let $g$ be a Riemannian metric on 3-sphere $S^{3}$ with positive sectional curvature and scalar curvature $R\geq 6$ . Given $\eta>0$, there exists $\varepsilon = \varepsilon(\eta)>0$ such that if the width of $g$ satisfies $W(g)>4\pi-\varepsilon$, then $g$ is $\eta$-close to a metric of constant curvature one in the $\mathcal{C}^{0}$-norm.
\end{theorem}
 
Theorem \ref{thmmain} similarly implies a rigidity statement for positively curved metrics $g$ on the real projective 3-space $RP^{3}$. In this case, following Bray-Brendle-Eichmair-Neves \cite{BBEN10}, we can consider the least area embedding of $RP^{2}$ inside $(RP^{3},g)$. We denote its area by $\mathcal{A}(g)$. In \cite{BBEN10}, the authors prove the following inequality:
$$\mathcal{A}(g)\inf R_{g} \leq 12\pi,$$
and showed that equality only happens for the round metric of ${RP}^{3}$. Using Theorem \ref{thmmain}, we show the almost equality case:

\begin{theorem}\label{thmapp2}
Let $g$ be a Riemannian metric on real projective 3-space $RP^{3}$ with positive sectional curvature and scalar curvature $R\geq 6$. Given $\eta>0$, there exists $\varepsilon = \varepsilon(\eta)>0$ such that if the least-area embedding of $RP^{2}$ satisfies $\mathcal{A}(g)>2\pi-\varepsilon$, then  then $g$ is $\eta$-close to a metric of constant curvature one in the $\mathcal{C}^{0}$-norm.
\end{theorem} 

The invariants $W(g)$ and $\mathcal{A}(g)$ can be thought of as a 2-dimensional analogs of the diameter of $g$. 
In this sense, it is interesting to compare theorems \ref{thmapp} and \ref{thmapp2} with previously known rigidity results for positively curved metrics with almost maximal diameter. 
For example, Colding \cite{Col96} proved that an $n$-dimensional Riemannian manifold with Ricci curvature at least $n-1$ and volume sufficiently close to the volume of the round $n$-sphere $\Sph^{n}$ must be close the round sphere $\Sph^{n}$ in the Gromov-Hausdorff distance.    
Moreover, Cheeger and Colding \cite{ChCo96} showed that an $n$-dimensional Riemannian manifold with Ricci curvature at least $n-1$ and almost maximal diameter must be close in the Gromov-Hausdorff sense to a metric suspension $(0,\pi) \times_{\sin r} X$ of some metric space $X$. 
However, there are smooth metrics on $S^{3}$, obtained by smoothings of the suspension over a small 2-sphere, that have sectional curvature bounded bellow by one, almost maximal diameter, and are {\it not} close to some round metric, even in the Gromov-Hausdorff sense. 
Hence, Theorem \ref{thmapp} and such examples indicate that the width is more rigid than the diameter.

\section{Proof of Theorem \ref{thmmain}}\label{sec:profmain}

In what follows, we consider a fixed Riemannian manifold $M^n$ and Ricci flows $g(t)$ on $M^n$ such that $R_{g(0)}\geq n(n-1)$ and $(M^{n},g(0))\times\mathbb{R}^{2}$ has positive isotropic curvature. 
We assume such flows exist at least for some time, say $\bar T_n$, where $\bar T_n\in\pl0,\frac{1}{2(n-1)}\pr$ is some number we fix for the remainder of the paper.
From now on, we will refer to a Ricci flow under the such hypothesis by {\it a Ricci flow as above}, and we denote the scalar curvature of the evolution of a constant curvature one metric by $\rho(t)=\frac{n(n-1)}{1-2(n-1)t}$. 

\begin{proposition}\label{prop:max}
Given any real number $0<\delta <1$, there exists $\tau =\tau (\delta)>0$  such that: if $g(t)$ is a Ricci flow as above with extinction time $T \geq \frac{1}{2(n-1)}-\tau$, then, for any time $0\leq t \leq \bar T_n$, the minimum of the scalar curvature at time $t$ over $M$ is less than $\rho(t)+\delta$.
\end{proposition}
\begin{proof}
Consider any such Ricci flow. By the maximum principle applied to evolution \eqref{eq:scalar} we have that $R(\cdot,t)\geq \rho(t)$. Moreover, if for a certain time $\overline{t}\in[0,\bar T_n)$ we have that 
$$R(\cdot,\overline{t})\geq \rho(\overline{t})+\delta,$$
then, again by the maximum principle and \eqref{eq:scalar} we must have for all $t \in [ \overline{t}, T)$ that:
$$R(\cdot,t)\geq \frac{\rho(\overline{t})+\delta}{1-2n^{-1}(\rho(\overline{t})+\delta)(t-\overline{t})}.$$
In particular, $T$ would satisfy $1-2n^{-1}(\rho(\overline{t})+\delta)(T-\overline{t})\geq0$, that is, 
\begin{equation}\label{eq:ext}
T\leq \frac{n}{2(\rho(\overline{t})+\delta)}+\overline{t} < \frac{n}{2 \rho ( \overline{t} )} + \overline{t} =\frac{1}{2(n-1)}.
\end{equation}
Since $\overline{t}\in[0,\bar T_n]$, inequality \eqref{eq:ext} gives us a contradiction if $T$ is sufficiently close to $\frac{1}{2(n-1)}$, so we are done.
\end{proof}

Note that Proposition \ref{prop:max} was a consequence of a (global) maximum principle applied to the evolution equation of the scalar curvature.
To obtain more information about such flows, we will need a localized version of the maximum principle, which comes in the form of Hamilton's Harnack inequality. 
More specifically, we will show that the set of points at which the scalar curvature attains a value close to $\rho(t)$ does not change too much in time.
We will prove: 

\begin{proposition}\label{prop:scalar} There exist a constant $A =A(n)< \infty$ such that for any time $t_2\in (0,\bar T_n]$, we can find a time $t_2/2 < t_1=t_{1}(t_{2},n) < t_2$ with the property that: given $ \theta >0 $ small, there exists $\delta=\delta(\theta, t_2, n)\in(0,1)$ such that for any Ricci flow as above, if one has a point $\overline{x}\in M$ satisfying $R(\bar{x},t_2)<\rho(t_2)+\delta$, then there must exist a point $y \in B\pl\overline{x}, t_1, A \sqrt{t_2 - t_1}\pr$ satisfying $R( y, t_1 ) < \rho (t_1) + \theta$.
\end{proposition}

\noindent{\bf Notation.} Given a Ricci flow $g(t)$,  a point in $x\in M$, and a number $r>0$, we write $B(x,t,r)$ to denote the open ball in $M$ centered in $x$ and of radius $r$ with respect to $g(t)$. If working with a sequence of flows $g_k(t)$, we write $B_k(x,t,r)$ to emphasize the dependence on $k$.\\

Before proving Proposition \ref{prop:scalar}, we would like to comment on some of the ingredients that go into its proof. First, recall Hamilton's Harnack inequality \cite{Ham93}, which was later generalized by Brendle \cite{Brendle:harnack} for Ricci flows with our curvature assumptions. Given two points $x,y\in M$ and times $0<t_1<t_2<T$, their inequality implies that the scalar curvature must satisfy:
\begin{equation}\label{eq:harnack}
R(y,t_1)\leq \frac{t_2}{t_1}\exp\pl \frac{d_{g(t_1)}^2(x,y)}{2(t_2-t_1)}\pr  R(x,t_2),
\end{equation}
where $d_{g(t_1)}(x,y)$ is the distance between the points $x,y$ measured with respect to the metric $g(t_1)$.
The Harnack inequality \eqref{eq:harnack} will be specially useful in our discussion since we can use it to propagate scalar curvature bounds backwards in time.
Moreover, since a Ricci flow as above has nonnegative sectional curvature everywhere, once one has a bound on the scalar curvature of a point, one immediately obtains bounds on all the sectional curvatures of that point, $i.e.$,
\begin{equation}\label{eq:curvaturebound}
 |\Rm|_{g(t)}(x,t) \leq R(x,t)
\end{equation} 
The other ingredient that goes into the proof of Proposition \ref{prop:scalar} is the existence of a suitable subsolution to the heat equation supported on small balls with good asymptotic behavior:

\begin{lemma}\label{lemma:sub}
There exists a constant  $A=A(n)>0$ such that: given any time $t_2\in (0,\bar T_n ]$, one can find an earlier time $t_2/ 2 <t_1=t_1(t_2,n)<t_2$ with the property that for any Ricci flow as above where one has a point $\overline{x}\in M$ satisfying $R(\bar{x},t_2)<\rho(t_2)+1$, there exists a smooth, bounded, nonnegative function $u(x,t)$ on $M \times [t_1,t_2)$ satisfying:

\begin{itemize}
\item[(i)] $\partial_t u \leq \Delta_{g(t)} u $ in the barrier sense, $i.e.,$ $u(x,t)$ is a subsolution of the heat equation;
\item[(ii)] for all $t\in[t_1,t_2)$, the function $u(x,t)$ is supported in the ball \linebreak[1] $B \big( \overline{x}, \linebreak[1] t, \linebreak[1] A\sqrt{t_2-t} \big)$
\item[(iii)] $0\leq u(\cdot,t)\leq 1$ and  $u(\bar{x},t)= \sqrt{t_{2}-t}$.
 \end{itemize}
 \end{lemma}
 
\begin{proof}
See Appendix.
\end{proof}

We are now ready to prove Proposition \ref{prop:scalar}:

\begin{proof}[Proof of Proposition \ref{prop:scalar}] Let $A$ and $t_1$ as in Lemma \ref{lemma:sub}. We argue by contradiction: Given $\theta>0$ small, suppose that for any $1>\delta>0$ it is possible to find a Ricci flow as above such that  $\rho(t_{2})\leq R(\overline{x},t_{2})<\rho(t_{2})+\delta$, but $R( y, t_1 ) > \rho (t_1) + \theta$ \text{\rm for all} $y \in B\pl\overline{x}, t_1, A \sqrt{t_2 - t_1}\pr$. Let $u(x,t)$ be the subsolution constructed in Lemma \ref{lemma:sub}. In such case, we get: 
$$R(\cdot,t_{1})>\rho(t_{1})+\theta u(\cdot,t_{1}).$$
Since
\begin{align}
 \partial_t \big( R (\cdot, t) - \theta u (\cdot, t) \big) &\geq \Delta \big( R(\cdot, t) - \theta u(\cdot, t) \big) + \frac{2}{n} R^2 \nonumber \\
 &\geq \Delta \big( R(\cdot, t) - \theta u(\cdot, t) \big) + \frac2{n} \big( R(\cdot, t) - \theta u (\cdot, t) \big)^2, \nonumber
\end{align}
we get by the maximum principle that also at latter times $t\in[t_{1},t_{2})$:
$$
R(\cdot,t)>\rho(t)+\theta u(\cdot,t).
$$
In particular, we get that at $\overline{x}$
$$ R(\overline{x}, t) > \rho(t) + \theta \sqrt{t_2 - t}. $$
Next, at the point $\bar{x}$, we have by Harnack's inequality (\ref{eq:harnack})
$$R(\bar{x},t)\leq \frac{t_2}{t} R(\bar{x},t_{2})<\frac{t_{2}}{t}(\rho(t_{2})+\delta),$$
so, for all $t \in [t_1, t_2)$
$$ \rho(t) + \theta \sqrt{t_2 - t} < \frac{t_2}t ( \rho (t_2) + \delta ). $$
Note that since $t_1 > t_2 /2$ we have $| \rho(t) - \frac{t_2}{t} \rho(t_2) | \leq C(n) (t_2 - t)$ for some positive dimensional constant $C(n)$, and hence we have for all $t \in [t_1, t_2)$:
$$ \theta \sqrt{t_2 - t} < C(n) (t_2 - t) + 2 \delta. $$
Choose now $t_{12} \in (t_1, t_2)$, $t_{12} = t_{12} (\theta, t_2, n )$ such that $C(n) (t_2 - t_{12}) < \frac12 \theta \sqrt{t_2 - t}$.
Then
$$ \tfrac12 \theta \sqrt{ t_2 - t_{12} } < 2 \delta. $$
So we obtain a contradiction if we set
\[ \delta = \delta (\theta, n, t_2 ) = \tfrac14 \theta \sqrt{ t_2 - t_{12}(\theta, t_2 , n) } .  \qedhere \]
 \end{proof}
 
We next prove that we have a uniform curvature bound on parabolic neighborhoods around points in $B (\bar{x}, t_1, A \sqrt{t_2 - t_1})$.
This is a preparatory step to using compactness.

\begin{lemma}\label{le:curvbd}
Let $g(t)$ be a Ricci flow as above with $\overline{x}\in M$ satisfying $R(\bar{x},t_2)<\rho(t_2)+1$, and $t_1,t_2$ and $A$ as in Proposition \ref{prop:scalar}. 
Given a positive number $D < \infty$, there exists a constant $C = C(D, t_1, t_2, n) < \infty$ such for any point $y\in B\pl\overline{x},t_1,A\sqrt{t_2-t_1}\pr$, we have a curvature bound 
$$|\Rm|_{g(t)}(x,t)<C $$
for all  $(x,t)\in B\pl y,t_1,D\pr\times[t_1/2,t_1].$
 \end{lemma}
 
\begin{proof}
Note that $x\in B\pl\overline{x},t_1, A \sqrt{t_2-t_1} + D \pr$.
As before, we can assume without loss of generality that $t_1>t_2/2$.
So we have by Harnack's inequality \eqref{eq:harnack}:
$$
 R(x,t_1)\leq 2 \exp\pl \frac{ (A \sqrt{t_2 - t_1} +D)^2}{2 (t_2 - t_1)} \pr (\rho(t_2)+1) =: C' (D, t_1, t_2, n).
$$
Again by Harnack's inequality (at $(x,t_1)$) we get that for $t \in [t_1 / 2, t_1]$
\[ R(x,t) \leq \frac{t_1}{t} R(x,t_1) \leq 2 R(x,t_1) \leq 2 C'. \]
The desired curvature bound follows using \eqref{eq:curvaturebound}.
\end{proof}

We are finally ready to show that if the scalar curvature is sufficiently close to $\rho (t_2)$ at some point $(\bar{x}, t_2)$, then the sectional curvatures must be pinched everywhere at some previous time.

\begin{proposition} \label{prop:epsilon}
Given $t_2 \in (0, \bar{T}_n]$ we can choose $t_2/2 < t_1 = t_1 (t_2, n) < t_2$ such that for every $\varepsilon>0$ there exists a  $\delta=\delta(\varepsilon,t_2,n)\in(0,1)$ with the following property: given any Ricci flow $g(t)$ as above for which there exist some $\overline{x}\in M$ with $R(\overline{x},t_2)<\rho(t_2)+\delta$, then: 
$$\left|\text{\rm Sec}_{g}(\cdot, t_1) -\frac{1}{n(n-1)}\rho(t_1) \right|<\varepsilon,$$
everywhere in $M$.
\end{proposition}

\begin{proof}
We first choose $t_1 = t_1 (t_2, n)$ according to Proposition \ref{prop:scalar} and for each positive integer $k$, let $\delta_{k}=\delta(1/k,t_{2},n)$ be as in Proposition \ref{prop:scalar} for $\theta=\frac{1}{k}$. 
Assume now that the lemma was false for some $\varepsilon > 0$.
Then we can find a sequence of Ricci flows $(M_k, g_k(t))$ as above and points $\overline{x}_k \in M_k$ such that $R(\overline{x}_k, t_2) < \rho (t_2) + \delta_k$, but such that the conclusion of the lemma fails for every $k$.
By Proposition \ref{prop:scalar} we can find points $y_k \in B_k(\bar{x}_k, t_1, A\sqrt{t_2 - t_1}) \subset M_k$ such that
$$ R( y_k, t_1 ) < \rho (t_1) + \frac{1}{k}. $$
Applying Lemma \ref{le:curvbd} for $D = 8 \pi (\rho(t_1) / 2n(n-1))^{-1/2}$ gives us that
\begin{equation} \label{eq:curvbound}
 |\Rm|_{g_k (t)} < C (D, t_1, t_2, n) \,\, \text{on} \,\, B_k(y_k, t_1, D) \times [t_1/2, t_1].
\end{equation}
In particular, the conjugacy radius on $B_k(y_k, t_1, \frac12 D)$ is bounded from below by some uniform constant $c = c(D, t_1, t_2, n) \in (0, \frac12 D )$.

We now prove the following claim.
\begin{Claim}
Consider a sequence $z_k \in B_k(y_k, t_1, \tfrac12 D)$ and assume that $R(z_k, t_1) \to \rho(t_1)$ as $k \to \infty$.
Then
$$ \sup_{B_k(z_k, t_1, c/2)} \left|\text{\rm Ric}_{g_k(t_1)}  -\frac{1}{n}\rho(t_1)g_k (t_1) \right| \to 0$$
and, in particular,
$$ \sup_{B_k(z_k, t_1, c/2)} \left| R_{g_k(t_1)}  - \rho(t_1) \right| \to 0$$
as $k \to \infty$.
\end{Claim}

\begin{proof}[Proof of Claim]
Let $\overline{g}_k (t)$ be the family of metrics that arise from pulling back $g_k(t)$ via the exponential map
$$ \exp_{z_k, g_k(t)} : B (0, c) \subset T_{z_k} M_k \longrightarrow M_k. $$
By (\ref{eq:curvbound}) and Shi's estimates, we conclude that $\overline{g}_k(t)$ smoothly subconverge to a limiting Ricci flow $\overline{g}_\infty (t)$, $t \in (t_1/2, t_1]$ on compact subsets.
This Ricci flow satisfies
$$ R_{\overline{g}_\infty(t) } \geq \rho(t) \qquad \text{on} \qquad M_k $$
$t \in (t_1/ 2, t_1]$ and
$$ R_{\overline{g}_\infty(t_1)} ( 0 ) = \rho (t_1). $$
By the strong maximum principle applied to (\ref{eq:scalar}), we then conclude that $\overline{g}_\infty (t_1)$ is Einstein with scalar curvature $\rho(t_1)$.
This proves the claim.
\end{proof}

Using the claim, we can now show by induction that for all $j \in \mathbb{N}$ with $j c / 4 < \frac12 D$ we have
$$ \sup_{B_k(z_k, t_1, j c/4)} \left|\text{\rm Ric}_{g_k(t_1)}  -\frac{1}{n}\rho(t_1)g_k(t_1) \right| \to 0$$
So we get that
$$ \sup_{B_k(z_k, t_1, \frac14 D )} \left|\text{\rm Ric}_{g_k(t_1)}  -\frac{1}{n}\rho(t_1)g_k(t_1) \right| \to 0$$
So for large enough $k$ we have $\textrm{Ric}_{g_k (t_1) } > \frac12 \cdot \frac1n \rho(t_1) g_k (t_1)$ on $B_k(z_k, t_1, \frac14 D )$
By Myer's Theorem this implies
\[ \textrm{diam}_{g_k (t_1)}  B(z_k, t_1, \tfrac14 D ) < \pi \pl \frac1{2n(n-1)} \rho (t_1) \pr^{-1/2} = \frac 18 D. \]
So $M_k= B_k(z_k, t_1, \frac14 D)$ for large $k$ and we get that
\begin{equation} \label{eq:Riccipinching} 
\sup_{M_k} \left|\text{\rm Ric}_{g_k(t_1)}  -\frac{1}{n}\rho(t_1)g_k(t_1) \right| \to 0. 
\end{equation}

Consider the universal covering $(\widetilde{M}_k, \widetilde{g_k} (t_1) )$ of $(M, g_k(t_1))$ and note that $\widetilde{M}_k$ is diffeomorphic to the standard sphere $S^n$.
We are now in a position to apply Theorem 0.4 in \cite{PeTu99}, and conclude that the injectivity radius of $(\widetilde{M}_k, \widetilde{g_k} (t_1) )$ is uniformly bounded from below, independently of $k$.
So, after passing to a subsequence, $(\widetilde{M}_k, \widetilde{g_k} (t_1) )$ converges to a smooth Riemannian manifold $(\widetilde{M}_\infty, \widetilde{g}_\infty)$ in the smooth Cheeger-Gromov sense.
Due to (\ref{eq:Riccipinching}), the metric $\widetilde{g}_\infty$ is Einstein.
So by \cite{Bre10}, $(\widetilde{M}_\infty, \widetilde{g}_\infty)$ has to be a symmetric space.
Since $\widetilde{M}_k$ is diffeomorphic $S^n$, we have $\widetilde{M}_\infty$ is diffeomorphic to $S^n$.
This implies that $(\widetilde{M}_\infty, \widetilde{g}_\infty)$ is the round sphere, which finishes the proof.
\end{proof}

\begin{proof}[\bf End of Proof of Theorem \ref{thmmain}] 
We argue by contradiction. 
By results of Brendle-Schoen \cite{BreSch09}, $M$ is diffeomorphic to a spherical space form.
Fix the topology $M$ and the constant $\eta > 0$ and suppose there exists a sequence of metrics $g_k$ satisfying the assumptions of Theorem \ref{thmmain} with extinction time $T_k$ approaching $\frac{1}{2(n-1)}$, but such that for any $k$ there is no metric $g$ of constant curvature one on $M$ with the property that $\Vert g_k - g \Vert_{\mathcal{C}^0} < \eta$.

Since $\lim_{k \to \infty} T_k = \frac1{2(n-1)}$, we can choose constants $0 < \delta_k < 1$ such that $\lim_{k \to \infty} \delta_k = 0$ and for large $k$ we have $T_k \geq \frac1{2(n-1)} - \tau (\delta_k)$, where $\tau$ is the constant from Proposition \ref{prop:max}.
Next, we can choose a sequence of times $t_{2, k} \in\pl0,T_{n}\pr$ and a sequence of constants $\varepsilon_k > 0$ such that $\lim_{k \to \infty} t_{2,k} = 0$ and $\lim_{k \to \infty} \varepsilon_k = 0$ and such that for large $k$ we have $\delta_k < \delta ( \varepsilon_k, t_{2,k}, n)$, where $\delta$ is the constant from Proposition \ref{prop:scalar}.

We can now apply Proposition \ref{prop:max} and find points $x_k\in M$ such that for large $k$:
$$R_{g_{k}}({x}_k,t_{2,k})<\rho(t_{2,k})+\delta_k.$$
By Proposition \ref{prop:epsilon}, there exist times $0 < t_{1,k} < t_{2,k}$ such that for large $k$ large enough we have:
$$
\left|\text{\rm Sec}_{g_{k}}(\cdot, t_{1,k}) -\frac{1}{n(n-1)}\rho(t_{1,k}) \right|<\varepsilon_k, 
$$
everywhere on $M$.
Since $\lim_{k \to \infty} t_{2,k} = 0$, we also have $\lim_{k \to \infty} t_{1,k} = 0$.

A standard argument, $e.g.$, Section 10.5.4 of \cite{Per06}, then implies that the metric $g_k(t_{1,k})$ must be $\mathcal{C}^{0}$-close to a metric of constant sectional curvature \linebreak[1] $\frac{1}{n(n-1)}\rho(t_{1,k}) \to 1$.
We will now compare the distance functions $d_{g_k (t)}$ at times $0$ and $t_{1,k}$.
By our curvature assumptions and Harnack's inequality:
$$
|\Rm_{g_k}| (\cdot ,t) \leq R_{g_k}(\cdot ,t) \leq \frac{t_{1,k}}{t} R_{g_k}(\cdot,t_{1,k}) \leq \frac{C(n)}{t}, 
$$
on $M \times (0, t_{1,k}]$.
Given any two points $x,y$ on $M$, by distance distortion estimates of Hamilton \cite{Ham95}:
\[d_{g_k(t)}(x,y)\geq d_{g_k(0)}(x,y) -C'(n) \int^t_0 \sqrt{\frac{1}{s}} ds,\]
which, together with the fact that distances shrink under under positive curvature, yields:
$$ d_{g_k(0)}(x,y) \geq d_{g_k(t_k)}(x,y) \geq d_{g_k(0)}(x,y) -C''(n) \sqrt{t_k}. $$
It follows that $(M, g_k)$ converges to a spherical space form of constant curvature one in the Gromov-Hausdorff sense.

Using Theorem 10.8.18 of \cite{BuBuIv01} (see also Remark 10.8.19) and the fact that $g_k$ has positive sectional curvature, we can improve the Gromov-Hausdorff convergence to $\mathcal{C}^0$-convergence, contradicting our assumptions on $g_k$ for large $k$.
\end{proof}

\section{Proof of Theorems \ref{thmapp} and \ref{thmapp2}}

Following Marques and Neves \cite{MaNe12}, we let $g$ flow by Ricci flow, with maximal existence say $T>0$, and for each $t\in [0,T)$, we calculate the min-max invariant $W(g(t))$. From \cite{MaNe12}, we know that:

\begin{itemize}
\item[(i)] $W(g(t)) \geq W(g(0))- 16\pi t$,
\item[(ii)] $\lim_{t\nearrow T} W(g(t))=0$. 
\end{itemize}
Hence, given $\eta > 0$ and assuming $W(g)>4\pi- 16 \pi \tau(\eta)$, we obtain $T > \frac{1}{4} - \tau(\eta)$ from (i) and (ii).
Therefore, the conclusion follows from Theorem~\ref{thmmain}.

Similarly, using the inequality $\mathcal{A}(g(t))\geq \mathcal{A}(g(0))-8\pi t$, $t\in[0,T]$, proved in \cite{BBEN10}, and the fact that $\lim_{t\nearrow T} \mathcal{A}(g(t))=0$, Theorem \ref{thmapp2} follows from Theorem \ref{thmmain}.

\section{Appendix}

Consider a Ricci flow as in Section \ref{sec:profmain} and let $t_2$ be some chosen time. Suppose $\overline{x}\in M$ is a point such that  $R (\overline{x}, t_2) < \rho (t_2) + 1 $ and let $A=A(n)> 1$ a real constant to be chosen. The Harnack inequality \eqref{eq:harnack} and inequality \eqref{eq:curvaturebound} imply that for all $t \in (t_2/2, t_2)$ we have the following curvature bounds
$$
|\Rm_{g(t)}|, \; |\text{\rm Ric}_{g(t)}| <K(A, n) \qquad \text{on} \qquad B\pl \overline{x},t, A\sqrt{t_2-t} \pr.
$$

We want to choose a time $t_2/ 2 <t_1=t_1(t_2,n)<t_2$, for which there will always exist a nonnegative function $u(x,t)$ on $[t_1,t_2)$ satisfying:
\begin{itemize}
\item[(i)]$\partial_t u \leq \Delta_{g(t)} u $ in the barrier sense, $i.e.,$ $u(x,t)$ is a subsolution of the heat equation;
\item[(ii)] for $t\in[t_1,t_2)$, the function $u(x,t)$ is supported in $B\pl \overline{x},t, A\sqrt{t_2-t}\pr$;
\item[(iii)] $0\leq u(\cdot,t)\leq 1$ and  $u(\bar{x},t)= \sqrt{t_{2}-t}$.
 \end{itemize}

Our Ansatz for $u$ is:
\[ u(x,t) = \sqrt{t_2-t} \, \varphi \pl{\alpha \frac{d^2_{g(t)} (\overline{x}, x)}{t_2-t} }\pr, \]
where $0 < \alpha < 1$ is a constant to be chosen and $\varphi : [0, \infty) \to [0, \infty)$ is a smooth real function with the following properties
\begin{itemize}
\item $\varphi \equiv 1$ on $[0,1]$,
\item $\varphi \geq \frac12$ on $[0,2]$,
\item $\varphi' \leq 0$ on $[0, \infty)$,
\item $\varphi'' \geq 0$ on $[2, \infty)$,
\item $\varphi |_{[3, \infty)} \equiv 0$.
\end{itemize}
It is clear that such a $\varphi$ can always be chosen.
Note that condition (iii) trivially holds.

We first argue that we can choose the constant $0 < \alpha < 1$ small enough such that for all $r \geq 0$
\begin{equation}\label{eq:varphi}
 - \frac12 \varphi (r) + \frac12  r \varphi' (r) \leq 4\alpha r \varphi'' (r) + 2 \alpha n  \varphi' (r), 
 \end{equation}
Indeed, on $[0,2]$ the left-hand side is bounded from above by $- \frac14$ and the absolute value of the right-hand side can be made arbitrarily small by choosing $\alpha$ small enough.
On the other hand, assuming $\alpha < \frac{1}{2n}$, we find that on $[2,3]$
\[ - \frac12 \varphi (r) + \frac12  r \varphi' (r) \leq \frac12 r \varphi' (r) \leq 2 \alpha n \varphi' (r) \leq 4\alpha r \varphi'' (r) + 2 \alpha n  \varphi' (r). \]
Lastly, on $[3, \infty)$ the inequality (\ref{eq:varphi}) trivially holds.
Fix $\alpha$ for the rest of the proof.

We now choose and fix $1 < A < \infty$ large enough such that $\alpha A^2 > 3$.
Then condition (ii) holds.
Note that the choice of $A$ also determines the constant $K= K(A,n)$.

It remains to choose $t_1$ close enough to $t_2$ such that condition (i) is satisfied, \textit{i.e.}, that $\partial_t u \leq \Delta_{g(t)} u$ in the barrier sense.
We start by computing the time-derivative. In what follows, we will always write $\tau = t_2-t$ and $d = d_{g(t)} (\overline{x}, x)$.
\begin{align*} \partial_t u(x,t) & = - \frac1{2\sqrt{\tau}} \varphi + \sqrt{\tau} \partial_t \Big( \alpha \frac{d^2}{\tau} \Big) \varphi'. \\
&= - \frac1{2\sqrt{\tau}} \varphi  + \alpha \sqrt{\tau} \frac{d^2}{\tau^2} \varphi' + 2 \alpha \sqrt{\tau} \frac{d \partial_t d}{\tau} \varphi' \\
& \leq  - \frac1{2\sqrt{\tau}} \varphi + \alpha \frac{d^2}{\tau^{3/2}} \varphi' - 2 \alpha K  \sqrt{\tau} \frac{d^2}{\tau} \varphi',
\end{align*}
where we are using that $\varphi'\leq0$ and $\partial_t d \geq - K d$.
We compute the gradient directly as follows:
\[ \nabla u (x,t) =  \sqrt{\tau}  2 \alpha \frac{d}{\tau} \varphi' \nabla_x d_{g(t)} (\overline{x}, x). \]
Using the fact that $|\nabla d| = 1$, $\text{\rm Ric}_{g(t)} > 0$ and $\varphi' \leq 0$, the Laplacian can be estimated as follows in the barrier sense:
\begin{align*}
\Delta u(x,t) &= \sqrt{\tau}  4 \alpha^2 \frac{d^2}{\tau^2} \varphi'' + \sqrt{\tau}  2 \alpha \frac{1}{\tau} \varphi' + \sqrt{\tau}  2 \alpha \frac{d}{\tau} \varphi' \Delta d_{g(t)} (x_0, x) \\
&\geq 4\alpha^2  \frac{d^2}{\tau^{3/2}} \varphi'' + 2\alpha  \frac{1}{\sqrt{\tau}} \varphi' +  \sqrt{\tau}  2 \alpha \frac{d}{\tau} \varphi' \frac{n-1}{d} \\
&\geq 4 \alpha^2  \frac{d^2}{\tau^{3/2}} \varphi'' + 2 \alpha  \frac1{\sqrt{\tau}} \varphi' + 2\alpha (n-1)  \frac1{\sqrt{\tau}}  \varphi' \\
 &= 4 \alpha^2  \frac{d^2}{\tau^{3/2}} \varphi'' + 2\alpha n  \frac1{\sqrt{\tau}}  \varphi'. 
\end{align*}
So, in order to satisfy (i), we must have
\[ - \frac1{2\sqrt{\tau}} \varphi + \alpha \frac{d^2}{\tau^{3/2}} \varphi' - 2 \alpha K  \sqrt{\tau} \frac{d^2}{\tau} \varphi'
\leq 4 \alpha^2  \frac{d^2}{\tau^{3/2}} \varphi'' + 2\alpha n  \frac1{\sqrt{\tau}}  \varphi' .  \]
We multiply by $\sqrt{\tau}$:
\[ - \frac1{2} \varphi + \alpha \frac{d^2}{\tau} \varphi' - 2 K  \tau  \alpha \frac{d^2}{\tau} \varphi'
\leq 4 \alpha^2  \frac{d^2}{\tau} \varphi'' + 2\alpha n\varphi' .  \]
Let us now assume that $\tau$ is small enough such that $2K \tau < \frac12$.
This assumption determines $t_1 = t_1 (t_2, n)$.
We therefore need to check whether
\[ - \frac1{2} \varphi + \frac12 \alpha \frac{d^2}{\tau} \varphi' 
\leq 4 \alpha^2  \frac{d^2}{\tau} \varphi'' + 2\alpha n\varphi' .  \]
This inequality follows from (\ref{eq:varphi}) for $r = \alpha \frac{d^{2}}{\tau}$.
So condition (i) holds for $t_1$ sufficiently close to $t_2$.

\bibliography{bib} 

\providecommand{\bysame}{\leavevmode\hbox to3em{\hrulefill}\thinspace}
\providecommand{\MR}{\relax\ifhmode\unskip\space\fi MR }
\providecommand{\MRhref}[2]{%
  \href{http://www.ams.org/mathscinet-getitem?mr=#1}{#2}
}
\providecommand{\href}[2]{#2}
\begin{thebibliography}{BBEN10}

\bibitem[BBEN10]{BBEN10}
H.~Bray, S.~Brendle, M.~Eichmair, and A.~Neves, \emph{Area-minimizing
  projective planes in 3-manifolds}, Comm. Pure Appl. Math. \textbf{63} (2010),
  no.~9, 1237--1247. \MR{2675487 (2011h:53068)}

\bibitem[BBI01]{BuBuIv01}
Dmitri Burago, Yuri Burago, and Sergei Ivanov, \emph{A course in metric
  geometry}, Graduate Studies in Mathematics, vol.~33, American Mathematical
  Society, Providence, RI, 2001. \MR{1835418 (2002e:53053)}

\bibitem[Bre09]{Brendle:harnack}
Simon Brendle, \emph{A generalization of {H}amilton's differential {H}arnack
  inequality for the {R}icci flow}, J. Differential Geom. \textbf{82} (2009),
  no.~1, 207--227. \MR{2504774 (2010d:53070)}

\bibitem[Bre10]{Bre10}
\bysame, \emph{Einstein manifolds with nonnegative isotropic curvature are
  locally symmetric}, Duke Math. J. \textbf{151} (2010), no.~1, 1--21.
  \MR{2573825 (2010m:53062)}

\bibitem[BS09]{BreSch09}
Simon Brendle and Richard Schoen, \emph{Manifolds with {$1/4$}-pinched
  curvature are space forms}, J. Amer. Math. Soc. \textbf{22} (2009), no.~1,
  287--307. \MR{2449060 (2010a:53045)}

\bibitem[CC96]{ChCo96}
Jeff Cheeger and Tobias~H. Colding, \emph{Lower bounds on {R}icci curvature and
  the almost rigidity of warped products}, Ann. of Math. (2) \textbf{144}
  (1996), no.~1, 189--237. \MR{1405949 (97h:53038)}

\bibitem[Col96]{Col96}
Tobias~H. Colding, \emph{Shape of manifolds with positive {R}icci curvature},
  Invent. Math. \textbf{124} (1996), no.~1-3, 175--191. \MR{1369414
  (96k:53067)}

\bibitem[Ham93]{Ham93}
Richard~S. Hamilton, \emph{The {H}arnack estimate for the {R}icci flow}, J.
  Differential Geom. \textbf{37} (1993), no.~1, 225--243. \MR{1198607
  (93k:58052)}

\bibitem[Ham95]{Ham95}
\bysame, \emph{The formation of singularities in the {R}icci flow}, Surveys in
  differential geometry, {V}ol.\ {II} ({C}ambridge, {MA}, 1993), Int. Press,
  Cambridge, MA, 1995, pp.~7--136. \MR{1375255 (97e:53075)}

\bibitem[KL13]{KoLa13}
Herbert Koch and Tobias Lamm, \emph{Parabolic equations with rough data}.

\bibitem[MN12]{MaNe12}
Fernando~C. Marques and Andr{{\'e}} Neves, \emph{Rigidity of min-max minimal
  spheres in three-manifolds}, Duke Math. J. \textbf{161} (2012), no.~14,
  2725--2752. \MR{2993139}

\bibitem[Pet06]{Per06}
Peter Petersen, \emph{Riemannian geometry}, second ed., Graduate Texts in
  Mathematics, vol. 171, Springer, New York, 2006. \MR{2243772 (2007a:53001)}

\bibitem[PT99]{PeTu99}
A.~Petrunin and W.~Tuschmann, \emph{Diffeomorphism finiteness, positive
  pinching, and second homotopy}, Geom. Funct. Anal. \textbf{9} (1999), no.~4,
  736--774. \MR{1719602 (2000k:53031)}

\bibitem[Sim02]{Sim02}
Miles Simon, \emph{Deformation of {$C^0$} {R}iemannian metrics in the direction
  of their {R}icci curvature}, Comm. Anal. Geom. \textbf{10} (2002), no.~5,
  1033--1074. \MR{1957662 (2003j:53107)}

\end{thebibliography}
\bibliographystyle{amsalpha}

\end{document}